\documentclass[12pt]{amsart}
\usepackage{amsmath}
\usepackage{amssymb}
\usepackage{amsthm}
\usepackage{amscd}
\usepackage{enumerate}
\usepackage{enumitem}
\usepackage{verbatim}
\usepackage{mathrsfs}
\usepackage{mathabx}
\usepackage{url}

\theoremstyle{plain}
\newtheorem{thm}{Theorem}[section]
\newtheorem{prop}[thm]{Proposition}
\newtheorem{lem}[thm]{Lemma}
\newtheorem{cor}[thm]{Corollary}

\theoremstyle{definition}
\newtheorem{dfn}[thm]{Definition}

\newcommand{\bC}{\mathbb{C}}

\newcommand{\bF}{\mathbb{F}}

\newcommand{\bZ}{\mathbb{Z}}

\newcommand{\frn}{\mathfrak{n}}

\makeatletter
    
    \@addtoreset{equation}{section}
\makeatother

\makeatletter
\renewcommand{\p@enumii}{}
\makeatother

\begin{document}

\title[Dimension variation for Drinfeld cuspforms of level $\Gamma_1(t)$]{Dimension variation of Gouv\^{e}a-Mazur type 
for Drinfeld cuspforms of level $\Gamma_1(t)$}
\author{Shin Hattori}
\address[Shin Hattori]{Department of Natural Sciences, Tokyo City University}
%\email{shin-h@math.kyushu-u.ac.jp}
%\affil{Faculty of Mathematics, Kyushu University}

\date{\today}

%\classification{11F52}
%\keywords{Drinfeld modular form, slope, variation}
%\thanks{Supported by JSPS KAKENHI Grant Numbers JP17K05177.}

\begin{abstract}
Let $p$ be a rational prime and $q>1$ a $p$-power. Let $S_k(\Gamma_1(t))$ be the space of Drinfeld cuspforms of level $\Gamma_1(t)$ and weight $k$ for $\bF_q[t]$. For any non-negative rational number $\alpha$, we denote by $d(k,\alpha)$ the dimension of the slope $\alpha$ generalized eigenspace for the $U$-operator acting on $S_k(\Gamma_1(t))$. In this paper, we prove a function field analogue of the Gouv\^{e}a-Mazur conjecture for this setting. Namely, we show that for any $\alpha\leq m$ and $k_1,k_2>\alpha+1$, if $k_1\equiv k_2 \bmod p^m$, then $d(k_1,\alpha)=d(k_2,\alpha)$.
\end{abstract}

\maketitle
%\tableofcontents

%---------------------------------------------------------------------

%---------------------------------------------------------------------

\section{Introduction}

Let $p$ be a rational prime, $q>1$ a $p$-power, $A=\bF_q[t]$ and $\wp\in A$ a monic irreducible polynomial. For 
$K_\infty=\bF_q((1/t))$, we denote by $\bC_\infty$ the $(1/t)$-adic completion of an algebraic closure of $K_\infty$. 
Then the Drinfeld upper half plane $\Omega=\bC_\infty\setminus K_\infty$ has a natural structure of a rigid analytic 
variety over 
$K_\infty$. 

Let $k$ be an integer and $\Gamma$ a subgroup of $\mathit{SL}_2(A)$. Then a Drinfeld modular form of level 
$\Gamma$ and weight $k$ is a rigid analytic function $f:\Omega\to \bC_\infty$ satisfying
\[
f\left(\frac{az+b}{cz+d} \right)=(cz+d)^kf(z)\text{ for any }z\in \Omega,\ \begin{pmatrix}
a&b\\c&d
\end{pmatrix} \in \Gamma
\]
and a holomorphy condition at cusps.
The notion of Drinfeld modular form can be considered as a function field analogue of that of elliptic modular form and 
the former often has properties which are parallel to the latter. However, despite that the theory of $p$-adic families 
of elliptic modular forms is highly developed and has been yielding many applications, $\wp$-adic properties of 
Drinfeld modular forms are not well-understood yet. A typical difficulty in the Drinfeld case seems that a na\"{i}ve 
analogue of the universal character $\bZ_p^\times\to \bZ_p[[\bZ_p^\times]]^\times$ is not locally analytic by 
\cite[Lemma 
2.5]{Jeong} and thus similar constructions to those in the classical case including \cite{AIP-SMF} will not immediately produce an analytic family of invertible sheaves interpolating automorphic line bundles.

Still, there seem to exist interesting structures in $\wp$-adic properties of Drinfeld modular forms. In \cite{BV1, 
BV0}, Bandini-Valentino studied an analogue of the classical Atkin $U$-operator, which we also denote by $U$, acting on 
the space $S_k(\Gamma_1(t))$ of Drinfeld cuspforms of level $\Gamma_1(t)$ and weight $k$. The operator $U$ is defined by
\begin{equation}\label{EqnDefU}
(Uf)(z)=\frac{1}{t}\sum_{\beta\in \bF_q}f\left(\frac{z+\beta}{t}\right).
\end{equation}
The normalized $t$-adic valuation of an eigenvalue of $U$ is called slope. Note that here we adopt the different 
normalization from that of Bandini-Valentino, and as a result our notion of slope is smaller than theirs by one.
For a non-negative rational number $\alpha$, we denote by $d(k,\alpha)$ the dimension of the generalized eigenspace of 
$U$ acting on $S_k(\Gamma_1(t))$ for the eigenvalues of slope $\alpha$. Then they proposed a conjecture on a $p$-adic 
variation of $d(k,\alpha)$ with respect to $k$ \cite[Conjecture 6.1]{BV0} which can be regarded as a function field 
analogue of the Gouv\^{e}a-Mazur conjecture \cite[Conjecture 1]{GM}. In this paper, we will prove it.

\begin{thm}(Theorem \ref{Main})\label{MainIntro}
	Let $m\geq 0$ be an integer and $\alpha$ a non-negative rational number. Suppose $\alpha\leq m$. Then the dimension $d(k,\alpha)$ of the slope $\alpha$ generalized eigenspace in $S_k(\Gamma_1(t))$ satisfies
	\[
	k_1,k_2> \alpha+1,\ k_1\equiv k_2\bmod p^m\Rightarrow d(k_1,\alpha)=d(k_2,\alpha).
	\]  
	\end{thm}

For the proof, put
\[
P^{(k)}(X)=\det(I-XU\mid S_k(\Gamma_1(t))).
\]
First note that, as is mentioned in \cite[\S4, Remarks]{Wan}, the arguments of \cite{GMU} and \cite{Wan} 
can be generalized over suitable Drinfeld modular curves (including $X^\Delta_1(\frn)$ of \cite{Ha_DMC}). In 
particular, the characteristic power series of $U$ acting on the spaces of $\wp$-adic overconvergent Drinfeld modular 
forms of weight $k_1$ and $k_2$ are congruent modulo $\wp^{p^m}$. Also in our setting, we can show the 
congruence $P^{(k_1)}(X)\equiv P^{(k_2)}(X)\bmod t^{p^m}$ up to some factor. However, though with this we can prove 
Theorem 
\ref{MainIntro} for $p\geq 3$, it is not enough to settle the case of $p=2$ on which Bandini-Valentino stated 
their conjecture.

Instead, we investigate the formula of the representing matrix of $U$ given by Bandini-Valentino \cite[(3.1)]{BV1} more closely. 
Luckily, the representing matrix is of very special form: each entry on the $j$-th column (with the normalization that the 
leftmost column is the zeroth) is an element of $\bF_q t^{j}$. Thanks to this fact, we can give a lower bound of 
elementary divisors of the representing matrix (Lemma \ref{LemGliss}). Then a perturbation argument shows that the 
$n$-th coefficients of $P^{(k)}(X)$ 
and $P^{(k+p^m)}(X)$ are much more congruent than modulo $t^{p^m}$ up to some factor of slope $\geq k-1$ (Corollary 
\ref{Perturb}), which is enough to yield Theorem 
\ref{MainIntro} for any $p$.

\subsection*{Acknowledgments} The author would like to thank Gebhard B\"{o}ckle for informing him of Valentino's table 
computing characteristic polynomials of $U$, and Maria Valentino for pointing out an error in the author's previous 
computer calculation. This work was supported by JSPS KAKENHI Grant Number JP17K05177.

%---------------------------------------------------------------------

%---------------------------------------------------------------------

\section{Dimension variation}

Let $k\geq 2$ be an integer. Put 
\[
\Gamma_1(t)=\left\{\gamma\in \mathit{GL}_2(A)\ \middle|\ \gamma\equiv \begin{pmatrix}1&*\\0&1\end{pmatrix}\bmod 
t\right\}\subseteq \mathit{SL}_2(A).
\]
On the space $S_k(\Gamma_1(t))$ of Drinfeld cuspforms of level $\Gamma_1(t)$ and weight $k$, we consider the 
$U$-operator for $t$ defined by (\ref{EqnDefU}).
Note that we follow the usual normalization of the $U$-operator which differs from that of \cite[\S 2.4]{BV1} by $1/t$. Then Bandini-Valentino \cite[(3.1)]{BV1} explicitly describe the action of $U$ with respect to some basis $\mathbf{c}_0^{(k)},\ldots, \mathbf{c}_{k-2}^{(k)}$, which reads as follows with our normalization:

\begin{equation}\label{EqnBV}
\begin{split}
U(\mathbf{c}_j^{(k)})=&(-t)^j\tbinom{k-2-j}{j}\mathbf{c}_j^{(k)}-t^j\sum_{h\in \bZ, h\neq 0}\left\{\tbinom{k-2-j-h(q-1)}{-h(q-1)}\right.\\
&\left.+(-1)^{j+1}\tbinom{k-2-j-h(q-1)}{j} \right\}\mathbf{c}_{j+h(q-1)}^{(k)}.
\end{split}
\end{equation}
Here it is understood that the binomial coefficient $\binom{c}{d}$ is zero if any of $c,d,c-d$ is negative and the 
terms involving $\mathbf{c}_{j+h(q-1)}^{(k)}$ are zero if $j+h(q-1)\notin 
[0,k-2]$. We denote by $U^{(k)}=(U^{(k)}_{i,j})_{0\leq i,j\leq k-2}$ the representing matrix of $U$ for this basis. 
Then we have $U^{(k)}\in M_{k-1}(A)$. We identify the $t$-adic completion of $A$ with $\bF_q[[t]]$ naturally and 
consider $U^{(k)}$ as an element of $M_{k-1}(\bF_q[[t]])$.

\begin{dfn}\label{DefGliss}
Let $B=(B_{i,j})_{0\leq i\leq m-1,0\leq j\leq n-1}$ be an element of $M_{m,n}(\bF_q[[t]])$. We say $B$ is glissando if $B_{i,j}\in \bF_q t^{j}$ for any $i,j$. 
\end{dfn}

By (\ref{EqnBV}), the matrix $U^{(k)}$ is glissando.

\begin{lem}\label{LemGliss}
	Let $B=(B_{i,j})_{0\leq i\leq m-1,0\leq j\leq n-1}$ be a glissando matrix in $M_{m,n}(\bF_q[[t]])$. Let $s_1\leq 
	s_2\leq\cdots \leq s_{r}$ be the elementary divisors of $B$ (namely, they are integers or $+\infty$ such that 
	$t^{s_i}$ is the $(i-1,i-1)$-entry of the Smith normal form of $B$). Then we have $s_l\geq l-1$ for any $l$.
	\end{lem}
\begin{proof}
	We prove the lemma by induction on $n$. For $n=1$, we have $s_1=0$ if $B\neq O$ and $s_1=+\infty$ otherwise. For 
	$n>1$, we may assume $B\neq O$ and let $c$ be the integer with $0\leq c\leq n-1$ such that the leftmost non-zero column 
	of $B$ is the $c$-th one. Since $B$ is glissando, the first elementary divisor of $B$ is $c\geq 0$ and the rest are 
	equal to the elementary divisors of a matrix $t^{c+1} B'$, where $B'$ is also glissando with $n-1$ columns. Let 
	$s'_1\leq \cdots\leq s'_{r'}$ be the elementary divisors of $B'$. By the induction hypothesis, we have $s'_l\geq 
	l-1$ and thus $s_l=c+1+s'_{l-1}\geq  l-1$ for $l\geq 2$. This concludes the proof.
	\end{proof}

Let $v_t$ be the $t$-adic additive valuation normalized as $v_t(t)=1$. For any element $P(X)=\sum_{n=0}^\infty p_n X^n\in \bF_q[[t]][[X]]$, the Newton polygon of $P(X)$ is by definition the lower convex hull of the set 
\[
\{(n,v_t(p_n))\mid n\geq 0\}.
\]

\begin{lem}\label{LemGlissSlope}
	Let $B\in M_m(\bF_q[[t]])$ be a glissando matrix. For any non-negative integer $l$, put 
	\[
	P(X)=\det(I-t^l XB)=\sum_{n=0}^m p_n X^n\in \bF_q[[t]][X]. 
	\]
	\begin{enumerate}
		\item\label{LemGlissSlope-n} $v_t(p_n)\geq ln+\frac{1}{2}n(n-1)$.
		\item\label{LemGlissSlope-s} Any slope of the Newton polygon of $P(X)$ is no less than $l$.
	\end{enumerate}
	\end{lem}
\begin{proof}
	First note that, for the characteristic polynomial 
	$Q(X)=\det (XI-t^l B)$, we have $P(X)=X^{m} Q(X^{-1})$ and thus $p_n$ is, up to a sign, 
	equal to the sum of the principal $n\times n$ minors of $t^l B$. Since $B$ is glissando, this shows (\ref{LemGlissSlope-n}). Since $p_0=1$, the resulting inequality $v_t(p_n)\geq ln$ implies (\ref{LemGlissSlope-s}).
	\end{proof}

Now we put
\[
P^{(k)}(X)=\det(I-XU^{(k)})=\sum_{n=0}^{k-1} a_n^{(k)} X^n
\]
and $a_n^{(k)}=0$ for any $n\geq k$.
Let $y=N^{(k)}(x)$ be the Newton polygon of $P^{(k)}(X)$. For any non-negative rational number $\alpha$, we denote by $d(k,\alpha)$ the dimension of the generalized 
eigenspace for the eigenvalues of normalized $t$-adic valuation $\alpha$. Then $d(k,\alpha)$ is equal to the width of 
the segment of slope $\alpha$ in the Newton polygon $N^{(k)}$.

\begin{lem}\label{LemOrd}
	$d(k,0)=1$.
	\end{lem} 
\begin{proof}
	By (\ref{EqnBV}), we have $U^{(k)}_{0,0}=\binom{k-2}{0}=1$. On the other hand, since $U^{(k)}$ is glissando, we have $v_t(U^{(k)}_{i,j})\geq j$ and 
	\[
	a_1^{(k)}=-\sum_{j=0}^{k-2} U^{(k)}_{j,j}\equiv -1\bmod t.
	\]
	Moreover, from Lemma \ref{LemGlissSlope} (\ref{LemGlissSlope-n}) we obtain $v_t(a_n^{(k)})>0$ for any $n\geq 2$. This yields the lemma. 
	\end{proof}

\begin{lem}\label{binomp}
	Let $a$ and $b$ be non-negative integers. Let $m\geq 1$ be an integer. Then we have
	\[
	\tbinom{a+p^m}{b}\equiv \tbinom{a}{b}+\tbinom{a}{b-p^m} \bmod p.
	\]
	Here it is understood that $\binom{c}{d}=0$ if any of $c,d,c-d$ is negative.
	\end{lem}
\begin{proof}
	This follows from 
	\[
	(X+1)^{a+p^m}\equiv (X+1)^a(X^{p^m}+1)\bmod p.
	\]
	\end{proof}

\begin{prop}\label{UMod}
	Let $m\geq 1$ be an integer. Then there exist glissando matrices $C\in M_{p^m,k-1}(A)$ and $D\in M_{p^m,p^m-k+1}(A)$ satisfying
	\[
	U^{(k+p^m)}\equiv \left(\begin{array}{c|c|c}U^{(k)} & O & O \\ C & t^{k-1} D& O\end{array}\right) \bmod t^{p^m}.
	\]
	Here it is understood that the middle blocks are empty if $p^m\leq k-1$.
	\end{prop}
\begin{proof}
	Let $j$ be an integer satisfying $0\leq j\leq k+p^m-2$. By (\ref{EqnBV}), the element $U(\mathbf{c}_j^{(k+p^m)})$ is equal to
	\begin{align*}
	(&-t)^j\tbinom{k+p^m-2-j}{j}\mathbf{c}_{j}^{(k+p^m)}\\
	&-t^j\sum_{h\in \bZ, h\neq 0}\left\{\tbinom{k+p^m-2-j-h(q-1)}{-h(q-1)}+(-1)^{j+1}\tbinom{k+p^m-2-j-h(q-1)}{j} \right\}\mathbf{c}_{j+h(q-1)}^{(k+p^m)}.
	\end{align*}
	
	Note that both of $U^{(k+p^m)}_{i,j}$ and $U^{(k)}_{i,j}$ are divisible by $t^{p^m}$ for $j\geq p^m$. Since $U^{(k+p^m)}$ is glissando, what we need to show is
	\begin{enumerate}
		\item\label{UModSmall} For any $j \leq \min\{k-2,p^m-1\}$ and $i\in [0,k-2]$, we have $U^{(k+p^m)}_{i,j}=U^{(k)}_{i,j}$, and
		\item\label{UModMiddle} If $k\leq p^m$, then for any $j\in [k-1,p^m-1]$ and $i\in [0,k-2]$, we have $U^{(k+p^m)}_{i,j}=0$.
		\end{enumerate}

	First we suppose $j\leq \min\{k-2,p^m-1\}$. By Lemma \ref{binomp}, the element $U(\mathbf{c}_j^{(k+p^m)})$ equals
	\begin{align*}
	&(-t)^j\left(\tbinom{k-2-j}{j}+\tbinom{k-2-j}{j-p^m}\right)\mathbf{c}_{j}^{(k+p^m)}\\
	&-t^j\sum_{\substack{h\in \bZ, h\neq 0\\j+h(q-1)\in [0,k-2]}}\left\{\tbinom{k-2-j-h(q-1)}{-h(q-1)}+\tbinom{k-2-j-h(q-1)}{-h(q-1)-p^m}\right.\\
	&\left.+(-1)^{j+1}\left(\tbinom{k-2-j-h(q-1)}{j}+ \tbinom{k-2-j-h(q-1)}{j-p^m}\right)\right\}\mathbf{c}_{j+h(q-1)}^{(k+p^m)}\\
	&-t^j\sum_{\substack{h\in \bZ, h\neq 0\\j+h(q-1)\geq k-1}}\left\{\tbinom{k+p^m-2-j-h(q-1)}{-h(q-1)}+(-1)^{j+1}\tbinom{k+p^m-2-j-h(q-1)}{j}\right\}\mathbf{c}_{j+h(q-1)}^{(k+p^m)}.\\
	\end{align*}
	Hence $U(\mathbf{c}_j^{(k+p^m)})$ agrees with
	\begin{align*}
	&\sum_{i=0}^{k-2}U^{(k)}_{i,j}\mathbf{c}_i^{(k+p^m)}+(-t)^j\tbinom{k-2-j}{j-p^m}\mathbf{c}_{j}^{(k+p^m)}\\
	&-t^j\sum_{\substack{h\in \bZ, h\neq 0\\j+h(q-1)\in [0,k-2]}}\left\{\tbinom{k-2-j-h(q-1)}{-h(q-1)-p^m}+(-1)^{j+1}\tbinom{k-2-j-h(q-1)}{j-p^m}\right\}\mathbf{c}_{j+h(q-1)}^{(k+p^m)}\\
	&-t^j\sum_{\substack{h\in \bZ, h\neq 0\\j+h(q-1)\geq k-1}}\left\{\tbinom{k+p^m-2-j-h(q-1)}{-h(q-1)}+(-1)^{j+1}\tbinom{k+p^m-2-j-h(q-1)}{j}\right\}\mathbf{c}_{j+h(q-1)}^{(k+p^m)}.\\
	\end{align*}
	
	Since $j<p^m$, we have $\binom{k-2-j}{j-p^m}=0$. For the case of $j+h(q-1)\in [0,k-2]$, we also have 
	$-h(q-1)-p^m\leq j-p^m<0$ and $\binom{k-2-j-h(q-1)}{-h(q-1)-p^m}=\binom{k-2-j-h(q-1)}{j-p^m}=0$. This proves (\ref{UModSmall}).
	
	Next we suppose $k\leq p^m$ and $j\in [k-1,p^m-1]$. For any $i\in [0,k-2]$, the element $U^{(k+p^m)}_{i,j}$ is equal to
	\[
	-t^j\left\{\tbinom{k+p^m-2-j-h(q-1)}{-h(q-1)}+(-1)^{j+1}\tbinom{k+p^m-2-j-h(q-1)}{j}\right\}
	\]
	if we can write $i=j+h(q-1)$ with some $h\neq 0$, and zero otherwise. Since $i\leq k-2$, we have $k-2-j-h(q-1)\geq 0$ and Lemma \ref{binomp} implies
	\begin{align*}
	\tbinom{k-2-j-h(q-1)+p^m}{-h(q-1)}&=\tbinom{k-2-j-h(q-1)}{-h(q-1)}+\tbinom{k-2-j-h(q-1)}{-h(q-1)-p^m},\\
	\tbinom{k-2-j-h(q-1)+p^m}{j}&=\tbinom{k-2-j-h(q-1)}{j}+\tbinom{k-2-j-h(q-1)}{j-p^m}.
	\end{align*}
	Since $i=j+h(q-1)\in [0,k-2]$ and $j<p^m$, we have 
	$\binom{k-2-j-h(q-1)}{-h(q-1)-p^m}=\binom{k-2-j-h(q-1)}{j-p^m}=0$ as is seen above. Since $j\geq k-1$, we also have 
	$\binom{k-2-j-h(q-1)}{-h(q-1)}=\binom{k-2-j-h(q-1)}{j}=0$. This proves (\ref{UModMiddle}) and the proposition follows.
\end{proof}

Let $V\in M_{k+p^m-1}(A)$ be the matrix of the right-hand side of Proposition \ref{UMod}. Let $D'$ be the upper $(p^m-k+1)\times (p^m-k+1)$ block of $D$ if $k\leq p^m$ and $D'=O$ otherwise. Then $D'$ is also glissando. Put
\[
\tilde{P}(X)=\det(I-XV)=P^{(k)}(X)\det(I-t^{k-1}XD')
\]
and write $\tilde{P}(X)=\sum_{n=0}^{k+p^m-1} \tilde{a}_n X^n$. We denote by $\tilde{N}$ the Newton polygon of $\tilde{P}(X)$.

\begin{cor}\label{Perturb}
	Let $m$ and $n$ be integers satisfying $m\geq 1$ and $0\leq n\leq k+p^m-1$. Then we have
	\[
	v_t(a_n^{(k+p^m)}-\tilde{a}_n)\geq p^m+\sum_{l=1}^{n-1}\min\{l-1, p^m\}.
	\]
	\end{cor}
\begin{proof}
Write 
\[
U^{(k+p^m)}=V+t^{p^m} W
\]
with some $W\in M_{k+p^m-1}(A)$. Let $s_1\leq\cdots\leq s_{k+p^m-1}$ be the 
elementary 
divisors of $V$. Since $V$ is glissando, by Lemma \ref{LemGliss} we obtain $s_l\geq l-1$ for any $l$. Then \cite[Theorem 
4.4.2]{Ked} shows
\[
v_t(a_n^{(k+p^m)}-\tilde{a}_n)\geq p^m+\sum_{l=1}^{n-1}\min\{s_l, p^m\}\geq p^m+\sum_{l=1}^{n-1}\min\{l-1, p^m\}.
\]
\end{proof}

\begin{lem}\label{LemNum}
	Let $m$ and $n$ be integers satisfying $m \geq 1$ and $n\geq 2$. Then we have
	\[
	p^m+\sum_{l=1}^{n-1}\min\{l-1, p^m\}>m(n-1).
	\]
	\end{lem}
\begin{proof}
	First we assume $n-2\geq p^m$. Then the left-hand side of the lemma is equal to
	\begin{equation}\label{EqnLarge}
	p^m+\sum_{l=1}^{p^m+1}(l-1)+\sum_{l=p^m+2}^{n-1} p^m=\frac{1}{2}p^m(p^m+3)+p^m(n-2-p^m).
	\end{equation}
	For $m\geq 1$, we have $\frac{1}{2}p^m\geq m$ and thus $\frac{1}{2}p^m(p^m+3)\geq m(p^m+2)$. Hence the right-hand side of (\ref{EqnLarge}) is greater than $m(n-1)$. 
	
	Next we assume $n-2< p^m$. In this case, the left-hand side of the lemma equals $p^m+\frac{1}{2}(n-1)(n-2)$. It is greater than $m(n-1)$ if and only if
	\[
	\left(n-\left(m+\frac{3}{2}\right)\right)^2+2p^m-m(m+1)-\frac{1}{4}> 0.
	\]
	Since $m$ and $n$ are integers, the first term is no less than $\frac{1}{4}$. Since we can show $2p^m>m(m+1)$ for any $p$ and $m\geq 1$, the lemma follows also for this case.
	\end{proof}

\begin{lem}\label{LemWindow}
	The part of the Newton polygon $\tilde{N}$ of $\tilde{P}(X)$ of slope less than $k-1$ agrees with that of $N^{(k)}$.
	\end{lem}
\begin{proof}
	For any $Q(X)\in \bF_q[[t]][X]$ and any non-negative rational number $\alpha$, the Newton polygon of $Q(X)$ has a 
	segment of slope $\alpha$ and width $l$ if and only if it has exactly $l$ roots of normalized $t$-adic 
	valuation $-\alpha$. By Lemma \ref{LemGlissSlope} (\ref{LemGlissSlope-s}), every root of the polynomial 
	$\det(I-t^{k-1}XD')$ has normalized $t$-adic valuation no more than $-(k-1)$. Thus, for $\tilde{P}(X)$ and 
	$P^{(k)}(X)$, the sets of roots of normalized $t$-adic valuation more than $-(k-1)$ agree including 
	multiplicities. This shows the lemma.
\end{proof}

\begin{thm}\label{Main}
	Let $k$ and $m$ be integers satisfying $k\geq 2$ and $m\geq 0$. Let $\alpha$ be a non-negative rational number satisfying $\alpha\leq m$ and $\alpha< k-1$. Then we have $d(k+p^m,\alpha)=d(k,\alpha)$.
	\end{thm}
\begin{proof}
	Let $\{\alpha_1,\ldots,\alpha_N\}$ be the set of slopes of the Newton polygons $N^{(k+p^m)}$ and $N^{(k)}$ which is no more than $m$ and less than $k-1$, and renumber them so that $\alpha_i< \alpha_{i+1}$ for any $i$. We proceed by induction, following the proof of \cite[Lemma 4.1]{Wan}. By Lemma \ref{LemOrd}, we have $\alpha_1=0$ and $d(k+p^m,0)=d(k,0)=1$. Thus we may assume $m\geq 1$ and $N\geq 2$. 
	
	Suppose that for some $r\leq N-1$, the equality $d(k+p^m,\alpha_i)=d(k,\alpha_i)$ holds for any $i$ satisfying $1\leq i\leq r$. By Lemma \ref{LemWindow}, this means that the Newton polygons $N^{(k)}$, $N^{(k+p^m)}$ and $\tilde{N}$ agree with each other on the part of slope no more than $\alpha_r$. Put $\alpha=\alpha_{r+1}>\alpha_1=0$ and let us prove $d(k+p^m,\alpha)=d(k,\alpha)$. We choose $k'\in \{k,k+p^m\}$ such that the slope $\alpha$ occurs in $N^{(k')}$ and let $k''$ be the other. Let $\beta$ be the slope of $N^{(k'')}$ on the right of $\alpha_r$. Then $\beta\geq \alpha$. 
	
	Let $(n,v_t(a_n^{(k')}))$ be the endpoint of the segment of $N^{(k')}$ of slope $\alpha$. Since the Newton polygon $N^{(k')}$ has a segment of slope zero, we have $n\geq 2$ and 
	\[
	N^{(k')}(n)=v_t(a_n^{(k')})\leq \alpha (n-1)\leq m(n-1).
	\]
	Then Corollary \ref{Perturb} and Lemma \ref{LemNum} imply 
	\begin{equation}\label{EqnDiff}
	v_t(a_n^{(k')})<v_t(a_n^{(k+p^m)}-\tilde{a}_n).
	\end{equation}
	
	If $k'=k$, then Lemma \ref{LemWindow} shows $v_t(a_n^{(k')})=v_t(a_n^{(k)})=v_t(\tilde{a}_n)$ and from (\ref{EqnDiff}) we obtain $v_t(a_n^{(k+p^m)})=v_t(\tilde{a}_n)=v_t(a_n^{(k)})$. This equality implies $\alpha=\beta$ and $d(k,\alpha)\leq d(k+p^m,\alpha)$. In particular, the slope $\alpha$ also occurs in $N^{(k+p^m)}$.

	If $k'=k+p^m$, then (\ref{EqnDiff}) gives $v_t(\tilde{a}_n)=v_t(a_n^{(k+p^m)})$. Let $\gamma$ be the slope of the 
	Newton polygon $\tilde{N}$ on the right of $\alpha_r$. Then this equality implies $\gamma\leq \alpha<k-1$. By Lemma 
	\ref{LemWindow}, we have $\beta=\gamma\leq \alpha$. Therefore, we have $\alpha=\beta=\gamma$ and the width of the 
	segment of slope $\alpha$ in $\tilde{N}$ is no less than that in $N^{(k+p^m)}$. Thus Lemma \ref{LemWindow} again 
	shows $d(k,\alpha)\geq d(k+p^m,\alpha)$. In particular, the slope $\alpha$ also occurs in 
	$N^{(k)}$. Combining these two cases, we obtain $d(k,\alpha)=d(k+p^m,\alpha)$. This concludes the proof of Theorem 
	\ref{Main}.
\end{proof}

%---------------------------------------------------------------------

%---------------------------------------------------------------------

\section{Remarks}

Computations using (\ref{EqnBV}) with Pari/GP indicate that the slopes appearing in $S_k(\Gamma_1(t))$ have some 
patterns (see also \cite[\S6]{BV0}). The below is a table of the case $p=q=2$, where the bold numbers denote 
multiplicities.

\begin{table}[h]
	\begin{center}
		\begin{tabular}{c|c}
			$k$&slopes\\\hline
			$2$& $0^{\mathbf{1}}$\\
			$3$& $0^{\mathbf{1}}, +\infty^{\mathbf{1}}$\\
			$4$& $0^{\mathbf{1}}, 1^{\mathbf{1}},+\infty^{\mathbf{1}}$\\
			$5$& $0^{\mathbf{1}}, \tfrac{3}{2}^{\mathbf{2}}, +\infty^{\mathbf{1}}$\\
			$6$& $0^{\mathbf{1}}, 1^{\mathbf{1}}, 2^{\mathbf{1}}, +\infty^{\mathbf{2}}$\\
			$7$& $0^{\mathbf{1}}, 2^{\mathbf{1}}, \tfrac{5}{2}^{\mathbf{2}}, +\infty^{\mathbf{2}}$\\
			$8$& $0^{\mathbf{1}}, 1^{\mathbf{1}}, 3^{\mathbf{3}}, +\infty^{\mathbf{2}}$\\
			$9$& $0^{\mathbf{1}}, \tfrac{3}{2}^{\mathbf{2}}, \tfrac{7}{2}^{\mathbf{2}}, +\infty^{\mathbf{3}}$\\
			$10$& $0^{\mathbf{1}}, 1^{\mathbf{1}}, 2^{\mathbf{1}}, 4^{\mathbf{3}}, +\infty^{\mathbf{3}}$\\
			$11$& $0^{\mathbf{1}}, 2^{\mathbf{1}}, 4^{\mathbf{1}}, \tfrac{9}{2}^{\mathbf{4}}, +\infty^{\mathbf{3}}$\\
			$12$& $0^{\mathbf{1}}, 1^{\mathbf{1}}, 3^{\mathbf{1}}, 4^{\mathbf{1}}, 5^{\mathbf{3}}, 
			+\infty^{\mathbf{4}}$\\
		\end{tabular}
		\begin{tabular}{c|c}
			$k$&slopes\\\hline
			$13$& $0^{\mathbf{1}}, \tfrac{3}{2}^{\mathbf{2}}, 4^{\mathbf{1}}, \tfrac{11}{2}^{\mathbf{4}}, 
			+\infty^{\mathbf{4}}$\\
			$14$& $0^{\mathbf{1}}, 1^{\mathbf{1}}, 2^{\mathbf{1}}, 5^{\mathbf{1}}, 6^{\mathbf{5}}, 
			+\infty^{\mathbf{4}}$\\
			$15$& $0^{\mathbf{1}}, 2^{\mathbf{1}}, \tfrac{5}{2}^{\mathbf{2}}, 6^{\mathbf{1}}, 
			\tfrac{13}{2}^{\mathbf{4}}, +\infty^{\mathbf{5}}$\\
			$16$& $0^{\mathbf{1}}, 1^{\mathbf{1}}, 3^{\mathbf{3}}, 7^{\mathbf{5}}, +\infty^{\mathbf{5}}$\\
			$17$& $0^{\mathbf{1}}, \tfrac{3}{2}^{\mathbf{2}}, \tfrac{7}{2}^{\mathbf{2}}, \tfrac{15}{2}^{\mathbf{6}}, 
			+\infty^{\mathbf{5}}$\\
			$18$& $0^{\mathbf{1}}, 1^{\mathbf{1}}, 2^{\mathbf{1}}, 4^{\mathbf{3}}, 8^{\mathbf{5}}, 
			+\infty^{\mathbf{6}}$\\
			$19$& $0^{\mathbf{1}}, 2^{\mathbf{1}}, 4^{\mathbf{1}}, \tfrac{9}{2}^{\mathbf{2}}, 8^{\mathbf{1}}, 
			\tfrac{17}{2}^{\mathbf{6}}, +\infty^{\mathbf{6}}$\\
			$20$& $0^{\mathbf{1}}, 1^{\mathbf{1}}, 3^{\mathbf{1}}, 4^{\mathbf{1}}, 5^{\mathbf{1}}, 8^{\mathbf{1}}, 
			9^{\mathbf{7}}, +\infty^{\mathbf{6}}$\\
			$21$& $0^{\mathbf{1}}, \tfrac{3}{2}^{\mathbf{2}}, 4^{\mathbf{1}}, \tfrac{11}{2}^{\mathbf{2}}, 
			8^{\mathbf{1}}, \tfrac{19}{2}^{\mathbf{6}}, +\infty^{\mathbf{7}}$\\
			$22$& $0^{\mathbf{1}}, 1^{\mathbf{1}}, 2^{\mathbf{1}}, 5^{\mathbf{1}}, 6^{\mathbf{1}}, 8^{\mathbf{1}}, 
			9^{\mathbf{1}}, 10^{\mathbf{7}}, +\infty^{\mathbf{7}}$\\
			$23$& $0^{\mathbf{1}}, 2^{\mathbf{1}}, \tfrac{5}{2}^{\mathbf{2}}, 6^{\mathbf{1}}, 8^{\mathbf{1}}, 
			10^{\mathbf{1}}, \tfrac{21}{2}^{\mathbf{8}}, +\infty^{\mathbf{7}}$\\
		\end{tabular}
	\end{center}
\end{table}

From the table, it seems that only small denominators are allowed for slopes: In the author's computation, as is already mentioned in \cite[\S1]{BV0}, the only 
case a non-trivial denominator appears is the case of $p=2$ and the denominator is at most $2$. Moreover, it seems 
likely that the finite slopes of $S_k(\Gamma_1(t))$ are less than $k-1$, and that for any $n$, the $n$-th smallest 
finite slope of $S_k(\Gamma_1(t))$ is bounded independently of $k$ (say, by $q^{n-1}$). If the 
latter observations hold in general, then combined with Theorem \ref{Main} it follows that for any $n$, the $n$-th 
smallest 
finite slopes of $S_k(\Gamma_1(t))$ are periodic of $p$-power period with respect to $k$ including multiplicities. For example, it 
seems from the table that the third smallest finite slopes of $S_k(\Gamma_1(t))$ in the case of $p=q=2$ are the repetition of
\[
2^{\mathbf{1}}, \tfrac{5}{2}^{\mathbf{2}}, 3^{\mathbf{3}},\tfrac{7}{2}^{\mathbf{2}}, 2^{\mathbf{1}}, 4^{\mathbf{1}}, 3^{\mathbf{1}},4^{\mathbf{1}}.
\]
This could be thought of as a function field analogue of Emerton's theorem \cite{Emerton} which asserts that the minimal 
slopes of $S_k(\Gamma_0(2))$ are periodic of period $8$.

%---------------------------------------------------------------------

%---------------------------------------------------------------------

%---------------------------------------------------------------------

%---------------------------------------------------------------------

%---------------------------------------------------------------------

%---------------------------------------------------------------------

\end{document}